\documentclass[a4paper,twoside,reqno]{amsart}

\usepackage[pagewise]{lineno}
\usepackage[utf8]{inputenc}

\usepackage{authblk}
\usepackage{hyperref}
\hypersetup{urlcolor=blue, citecolor=red}
\usepackage{amsmath,amsthm,amssymb,xcolor,bbm,mathrsfs,graphicx, float}
\usepackage[english]{babel}
\usepackage{fancyhdr}
\newcommand{\E}{\mathbb{E}}

\DeclareMathOperator{\card}{card}

\makeatletter
\@namedef{subjclassname@1991}{2020 Mathematics Subject Classification}
\makeatother

\title{The (n,k) game with heterogeneous agents}
\author{Hsin-Lun Li}

\date{}
\email{hsinlunl@math.nsysu.edu.tw}

\theoremstyle{definition}
\newtheorem{theorem}{Theorem}

\newtheorem{lemma}[theorem]{Lemma}

\begin{document}

\allowdisplaybreaks

\thispagestyle{firstpage}
\maketitle
\begin{center}
    Hsin-Lun Li
    \centerline{$^1$National Sun Yat-sen University, Kaohsiung 804, Taiwan}
\end{center}
\medskip

\begin{abstract}
    The \((n,k)\) game models a group of \(n\) individuals with binary opinions, say 1 and 0, where a decision is made if at least \(k\) individuals hold opinion 1. This paper explores the dynamics of the game with heterogeneous agents under both synchronous and asynchronous settings. We consider various agent types, including consentors, who always hold opinion 1, rejectors, who consistently hold opinion 0, random followers, who imitate one of their social neighbors at random, and majority followers, who adopt the majority opinion among their social neighbors. We investigate the likelihood of a decision being made in finite time. In circumstances where a decision cannot almost surely be made in finite time, we derive a nontrivial bound to offer insight into the probability of a decision being made in finite time.

\end{abstract}

\section{Introduction}

The \((n, k)\) game states that in a collection of \(n\) agents with binary opinions, such as agree/disagree or option A/option B, denoted as 1 and 0, a decision is made if at least \(k\) individuals hold opinion 1. Interpreting this mathematically, let \([n] = \{1, \ldots, n\}\) represent a collection of \(n\) agents, and let \(x_i(t)\) be the opinion of agent \(i\) at time \(t\). A decision is made if \(\sum_{i \in [n]} x_i(t) \geq k\) for some \(t \geq 0\). Some realistic examples of the \((n, k)\) game exist. For instance, a bill in a legislative body may go through several rounds until the proportion of individuals consenting reaches a specified threshold, such as \( 1/2 \) or \( 2/3 \). The threshold corresponds to $k/n$. There are quite a few spin models~\cite{castellano2009statistical}. One of the popular ones is the voter model~\cite{clifford1973model,holley1975ergodic}. In the voter model, an agent is uniformly selected to update its opinion by emulating one of its social neighbors, chosen uniformly at random. The threshold \( k \) in the \( (n, k) \) game differs from the threshold in the threshold voter model. In the threshold voter model, an agent adopts the opposing opinion if the number of its social neighbors holding the opposing opinion reaches the threshold~\cite{andjel1992clustering, cox1991nonlinear, durrett1992multicolor, durrett1993fixation, liggett1994coexistence}. In a nutshell, meeting the threshold in the \( (n, k) \) game indicates no further actions, while meeting the threshold in the threshold voter model implies altering one’s opinion. There are several ways to mimic others. For instance, an agent may alter its opinion if the opposing opinion becomes the majority among its social neighbors~\cite{li2024imitation}. In continuous opinion space, the opinion distance reflects the extent of difference in opinion between agents~\cite{mHK,mHK2,lanchier2020probability,lanchier2022consensus}. On the contrary, in spin opinion space, the opinion space is finite, and the relationship between agents' opinions is either identical or opposite.

The game is \emph{synchronous} if all agents update their opinions at each time step, and \emph{asynchronous} if only one agent, uniformly selected at random, updates its opinion at each time step. In this paper, under both the synchronous \((n, k)\) game and the asynchronous \((n, k)\) game, as well as with homogeneous or heterogeneous agents, we investigate
\begin{itemize}
    \item whether a decision can be made in finite time, and
    \item how fast a decision can be made.
\end{itemize}
Denote \( x_i(t) \sim \text{Bernoulli}(p) \) to indicate that \( x_i(t) \) is a Bernoulli random variable with \( P(x_i(t) = 1) = p \), meaning that \( x_i(t) = 1 \) with probability \( p \). An agent is a \emph{rejector} if \( x_i(t) = 0 \) at all times, a \emph{consentor} if \( x_i(t) = 1 \) at all times, and a \emph{neutralist} if \( x_i(t) \sim \text{Bernoulli}(1/2) \) at all times. With an initial opinion distributed as Bernoulli(1/2), an agent is called a \emph{random follower} if it emulates one of its social neighbors, selected uniformly at random, at the next time step; a \emph{majority follower} if it imitates the majority opinion among its social neighbors at the next time step; and a \emph{minority follower} if it mimics the minority opinion among its social neighbors at the next time step. A \emph{state} of all individuals refers to the opinion distribution among all individuals. For instance, \( (x_i(t))_{i \in [n]} \) represents the state of all agents at time \( t \). We can interpret the social relationships using an undirected simple graph, where a vertex represents an individual and an edge symbolizes the existing social connection between the corresponding agents. In the $(n, k)$ game, we assume that:
\begin{itemize}
    \item the social graph is complete, i.e., all agents are socially connected to each other.
\end{itemize}
For instance, Figure~\ref{fig:K5} shows the complete graph of order $5$. The set of social neighbors of agent $i$ is $[5] \setminus \{i\}$.
\begin{figure}[h!]
    \centering
    \includegraphics[width=5cm]{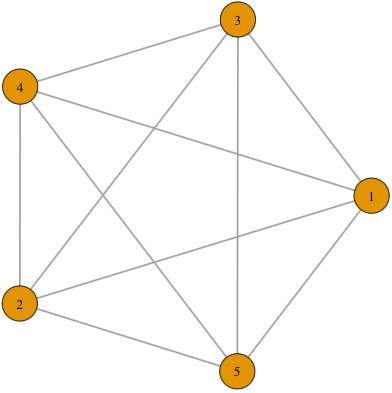}
    \caption{Complete graph of order 5}
    \label{fig:K5}
\end{figure}

\section{Main results}
In the asynchronous $(n,k)$ game consisting of \(n_r\) rejectors and \(n - n_r\) random followers, it is clear that a decision cannot be made in finite time if the threshold \(k > n - n_r\). We derive a nontrivial upper bound for the probability of a decision being made in finite time when the threshold \(k \leq n - n_r\).

\begin{theorem}\label{Thm:random follower}
    In the asynchronous \((n,k)\) game consisting of \(n_r\) rejectors and \(n - n_r\) random followers, the probability of a decision being made in finite time is at most \((n - n_r) / (2k)\) for the threshold \(k \leq n - n_r\).
\end{theorem}
If \( n_r \approx n/2 \) and \( k \approx n/2 \), then
\[
\frac{n - n_r}{2k} \approx \frac{n/2}{n} = \frac{1}{2}.
\]
In this case, the probability of a decision being made in finite time is at most \( 1/2 \) as $n\to\infty$. In other words, the probability of a decision not being made in finite time is at least \( 1/2 \) as $n\to\infty$. In the asynchronous $(n,k)$ game consisting of \(n_r\) rejectors, \(n_c\) consentors, and at least two majority followers, it is clear that a decision can be made in finite time when the threshold \(k \leq n_c\) and cannot be made in finite time when the threshold \(k > n - n_r\). We obtain a nontrivial upper bound for the probability of a decision not being made in finite time when the threshold \(n_c < k \leq n - n_r\).

\begin{theorem}\label{Thm:majority follower}
    In the asynchronous $(n, k)$ game consisting of \(n_r\) rejectors, \(n_c\) consentors and at least two majority followers with the threshold \(n_c < k \leq n - n_r\), the probability of a decision not being made in finite time is at most
$$
\frac{(n - n_c - n_r)(n + n_c + n_r - 1) + 4n_c n_r }{4n_c(n - n_c)}.
$$
\end{theorem}

If \( n_r = 0 \) and \( n_c \approx n/2 \), then the probability of a decision not being made in finite time is at most \( 3/4 \) as \( n \to \infty \). Namely, the probability of a decision being made in finite time is at least \( 1/4 \) as \( n \to \infty \).

\section{Properties of the game with heterogeneous agents}
Suppose the group of all agents consists of rejectors and agents with a positive probability of holding opinion 1 at all times. In that case, a decision can almost surely be made in finite time if and only if the number of non-rejectors is at least \( k \). Note that a consentor's opinion can be viewed as Bernoulli\((1)\).

\begin{lemma}\label{lemma: nonrejector}
    In the synchronous \((n, k)\) game with \( n_r \) rejectors, let \( [n-n_r] \) consist of the remaining agents, each satisfying \( x_i(t) \sim \text{Bernoulli}(p_i) \) with \( p_i > 0 \) for all \( t \geq 0 \). Then, a decision can almost surely be made in finite time if and only if \( k \leq n-n_r \).
\end{lemma}

\begin{proof}
    It is clear that a decision cannot be made if \( n - n_r < k \). When \( n - n_r \geq k \), the absorbing states are those where the number of agents holding opinion 1 is at least \( k \). Since the transition from non-absorbing states to absorbing states has a positive probability, the non-absorbing states are transient. By the theory of Markov chains, all states converge to an absorbing state in finite time.
\end{proof}

Next, we investigate how fast a decision can be made. Considering the synchronous game consisting of $n_r$ rejectors, $n_c$ consentors and $n_n$ neutralists, let $T$ be the earliest time that a decision is made. It is clear that $T=0$ if $n_c\geq k.$  For $n_c<k\leq n-n_r$, the problem reduces to how fast threshold $k-n_c$ can be reached among $n_n$ neutralists. It turns out that for large $n_n$, $T$ is a geometric random variable with success probability $p = 1 - \Phi\left(\frac{k - n_c - n_n/2}{\sqrt{n_n}/2}\right)$, where $\Phi(z) = P(\text{N}(0,1) \leq z)$ for $\text{N}(0,1)$ a standard normal random variable, i.e., a normal random variable with mean 0 and variance 1. So the expected value of $T$, $$\E [T]=\sum_{k\geq 0}kp(1-p)^k=\frac{1-p}{p}.$$ This means that the average number of rounds to decide is \( \E[T] + 1 = \frac{1}{p} \) for large \( n_n \). The average number of rounds to decide is 2 for large \( n_n \) when \( n_c = 0 \) and \( k \approx \frac{n_n}{2} \), i.e., when there are zero consentors and the threshold \( k \) approximates half of the neutralists. Similarly, in Lemma~\ref{lemma: nonrejector}, we can estimate that the average number of rounds to decide is between \( 1/p_{\max} \) and \( 1/p_{\min} \) for large \( n-n_r \), where
\begin{align*}
    p_{\max} &= 1 - \Phi\left(\frac{k - (n-n_r)\max_{i\in [n-n_r]}p_i}{\sqrt{(n-n_r)\max_{i\in [n-n_r]}p_i(1-\max_{i\in [n-n_r]}p_i)}}\right),\\
    p_{\min} &= 1 - \Phi\left(\frac{k - (n-n_r)\min_{i\in [n-n_r]}p_i}{\sqrt{(n-n_r)\min_{i\in [n-n_r]}p_i(1-\min_{i\in [n-n_r]}p_i)}}\right).
\end{align*}
Thus, we derive the following lemma.

\begin{lemma}
    In the synchronous \((n, k)\) game with \( n_r \) rejectors, let \( [n-n_r] \) consist of the remaining agents, each satisfying \( x_i(t) \sim \text{Bernoulli}(p_i) \) with \( p_i > 0 \) for all \( t \geq 0 \). For \( k \leq n-n_r \), the average number of rounds to decide is between \( 1/p_{\max} \) and \( 1/p_{\min} \) for large \( n-n_r \).
\end{lemma}

Next, we study the agent formation consisting of \(n_r\) rejectors, \(n_c\) consentors and \(n - n_r - n_c\) random followers. It is clear that a decision can be made if \(n_c \geq k\) and cannot be made if \(n - n_r < k\). Thus, we focus on the case where \(n_c < k \leq n - n_r\). In this scenario, the presence of a consentor ensures that a decision can almost surely be made in finite time.

\begin{lemma}
    In the asynchronous $(n,k)$ game with $n_r$ rejectors and $n_c$ consentors, let $[n-n_r-n_c]$ consist of the remaining agents who are random followers. For $1\leq n_c<k\leq n-n_r$, a decision can almost surely be made in finite time.
\end{lemma}

\begin{proof}
    The absorbing states are those where the number of agents holding opinion 1 is at least \( k \). Since all other states are transient, according to the Markov chains theory, a decision can almost surely be made in finite time.
\end{proof}

When there is no consentor, the state where all agents hold opinion 0 becomes an additional absorbing state. Thus, a decision cannot almost surely be made in finite time. We construct a supermartingale to derive the probability of a decision being made in finite time.

\begin{lemma}\label{lemma:supermartingale}
    In the asynchronous \((n,k)\) game consisting of \(n_r\) rejectors and \(n - n_r\) random followers, let \( Z_t = \sum_{i \in [n]} x_i(t) \). Then,
\[
\mathbb{E} [Z_{t+1} - Z_t] = -\frac{n_r \E[Z_t]}{n(n-1)},
\]
so \( (Z_t)_{t \geq 0} \) is a supermartingale.
\end{lemma}

\begin{proof}
    Observe that $$\E [Z_{t+1}-Z_t]=\E\left[ \frac{n-Z_t-n_r}{n}\frac{Z_t}{n-1}-\frac{Z_t}{n}\frac{n-Z_t}{n-1}\right]=-\frac{n_r \E[Z_t]}{n(n-1)}\leq 0$$ so $(Z_t)_{t\geq 0}$ is a supermartingale.
\end{proof}

Observe that \( Z_t \) is the number of individuals holding opinion 1 at time \( t \). Unlike the voter model, where \( (Z_t) \) is a martingale, \( (Z_t) \) is a supermartingale in the asynchronous \((n,k)\) game consisting of rejectors and random followers. Without rejectors, \( (Z_t) \) is a martingale.

\begin{proof}[\bf Proof of Theorem~\ref{Thm:random follower}]
    Let \( T \) be the earliest time at which no agents alter their opinions afterward, i.e., 
\[ T = \inf\{t \geq 0 : x_i(t) = x_i(s) \text{ for all } i \in [n] \text{ and } s \geq t\}. \]
The absorbing states are those where the number of agents holding opinion 1 reaches the threshold \( k \), and the state where all agents hold opinion 1. All other states are transient, so by the theory of Markov chains, \( T \) is almost surely finite and either \( Z_T \geq k \) or \( Z_T = 0 \). It follows from Lemma~\ref{lemma:supermartingale} that \( (Z_t)_{t \geq 0} \) is a nonnegative supermartingale. By the optional stopping theorem and conditional expectation, we have
\[ \mathbb{E}[Z_T \mid Z_T \geq k] \cdot P(Z_T \geq k) = \mathbb{E}[Z_T] \leq \mathbb{E}[Z_0], \]
so
\[ P(Z_T \geq k) \leq \frac{\mathbb{E}[Z_0]}{\mathbb{E}[Z_T \mid Z_T \geq k]} \leq \frac{(n - n_r) / 2}{k} = \frac{n - n_r}{2k}. \]

\end{proof}

Next, we investigate the agent formations of rejectors, consentors and majority followers, and those of rejectors, consentors and minority followers.

\begin{lemma}\label{lemma:W_t}
    In the asynchronous $(n,k)$ game, let $W_t=\sum_{i,j\in [n]}\mathbbm{1}\{x_i(t)\neq x_j(t)\}$ and $k_t$ be the agent uniformly selected at time $t.$ Then,
    \begin{align*}
        &W_t-W_{t+1}\\
        &\hspace{0.5cm}=2\sum_{j\in [n]-\{k_t\}}\left(\mathbbm{1}\{x_{k_t}(t)\neq x_j(t)\}-\mathbbm{1}\{x_{k_t}(t)= x_j(t)\}\right)\mathbbm{1}\{x_{k_t}(t)\neq x_{k_t}(t+1)\}.
    \end{align*}
\end{lemma}

\begin{proof}
    Observe that 
    \begin{align*}
        &W_t-W_{t+1}=2\sum_{j\in [n]-\{k_t\}}\left(\mathbbm{1}\{x_{k_t}(t)\neq x_j(t)\}-\mathbbm{1}\{x_{k_t}(t+1)\neq x_j(t)\}\right)\\
        &\hspace{0.5cm}=2\sum_{j\in [n]-\{k_t\}}\left(\mathbbm{1}\{x_{k_t}(t)\neq x_j(t)\}-\mathbbm{1}\{x_{k_t}(t)= x_j(t)\}\right)\mathbbm{1}\{x_{k_t}(t)\neq x_{k_t}(t+1)\}.
    \end{align*}
\end{proof}

Considering the asynchronous \((n,k)\) game with rejectors and consentors, \((W_t)\) is a supermartingale when the remaining agents are majority followers, and a submartingale when the remaining agents are minority followers. It turns out that the finite-time convergence of all agents' opinions holds under these circumstances.

\begin{lemma}\label{lemma:T finite}
    In the asynchronous \((n,k)\) game with rejectors and consentors, let \( T = \inf\{t \geq 0 : x_i(t) = x_i(s) \text{ for all } i \in [n] \text{ and } s \geq t\} \). Then, \( T \) is almost surely finite if 
\begin{itemize}
    \item the remaining agents are majority followers, or
    \item the remaining agents are minority followers.
\end{itemize}
\end{lemma}

\begin{proof}
    By Lemma~\ref{lemma:W_t}, \( W_t \) and \( -W_t \) are bounded supermartingales when the remaining agents are majority followers and minority followers, respectively. By the martingale convergence theorem, \( W_t \) converges almost surely to some random variable \( W_\infty \) with finite expectation. Since \( W_t - W_{t+1} \in \mathbb{Z} \) converges to 0 as \( t \to \infty \), \( (W_t)_{t \geq 0} \) converges in finite time. Additionally, since each agent in \( [n] \) is selected with positive probability, \( (x_i(t))_{t \geq 0} \) converges in finite time for all \( i \in [n] \).
\end{proof}
Either a decision is made at time \( T \) or it cannot be made in finite time. It turns out that all majority followers hold opinion 0 at time \( T \) when a decision cannot be made in finite time in the asynchronous $(n,k)$ game consisting of \(n_r\) rejectors, \(n_c\) consentors and at least two majority followers with the threshold \(n_c < k \leq n - n_r\). It is clear that a decision can be made when \(k \leq n_c\) and cannot be made when \(k > n - n_r\).

\begin{lemma}\label{lemma:opinion 0 in MF}
    In the asynchronous $(n, k)$ game consisting of \(n_r\) rejectors, \(n_c\) consentors and at least two majority followers with the threshold \(n_c < k \leq n - n_r\), all majority followers hold opinion 0 at time \( T \) when a decision cannot be made in finite time.
\end{lemma}

\begin{proof}
    Assume by contradiction that two majority followers hold distinct views at time \( T \), say agents \( i \) and \( j \), with opinions 1 and 0 at time \( T \), respectively. Then, 
\begin{align*}
    & \card(\{k \in [n] - \{i,j\} : x_k(T) = 1\}) \geq \lceil (n-1)/2 \rceil, \\
    & \card(\{k \in [n] - \{i,j\} : x_k(T) = 0\}) \geq \lceil (n-1)/2 \rceil.
\end{align*}
So, 
\begin{align*}
    n &= \card(\{k \in [n] - \{i,j\} : x_k(T) = 1\}) + \card(\{k \in [n] - \{i,j\} : x_k(T) = 0\}) + 2 \\
    & \geq 2(n-1)/2 + 2 = n + 1, \ \text{a contradiction}.
\end{align*}
Thus, all majority followers hold opinion 0 at time $T$ if a decision cannot be made in finite time.
\end{proof}

\begin{proof}[\bf Proof of Theorem~\ref{Thm:majority follower}]
    Since \( (W_t) \) is a bounded supermartingale and \( T \) is almost surely finite, by the optional stopping theorem, 
\(
    \mathbb{E}[W_T] \leq \mathbb{E}[W_0].
\)
By Lemma~\ref{lemma:opinion 0 in MF} and conditional expectation,
\[
    \mathbb{E}[W_T \mid W_T = 2(n - n_c)n_c]P(W_T = 2(n - n_c)n_c) \leq \mathbb{E}[W_T] \leq \mathbb{E}[W_0],
\]
where 
\begin{align*}
    \mathbb{E}[W_0] &= 2 \frac{1}{2} \binom{n - n_c - n_r}{2} + 2n_c n_r + 2 \frac{1}{2} n_c(n - n_c - n_r) + 2  \frac{1}{2} n_r(n - n_c - n_r) \\
    &= \frac{1}{2}[(n - n_c - n_r)(n + n_c + n_r - 1) + 4n_c n_r],
\end{align*}
and \( P(W_T = 2n_c(n - n_c)) \) is the probability of a decision not being made in finite time.

Hence,
\[
    P(W_T = 2(n - n_c)n_c) \leq \frac{(n - n_c - n_r)(n + n_c + n_r - 1) + 4n_c n_r}{4(n - n_c)n_c}.
\]
\end{proof}

\section{Statements and Declarations}
\subsection{Competing Interests}
The author is funded by NSTC grant.

\subsection{Data availability}
No associated data was used.

\end{document}